\documentclass[12pt,a4paper]{amsart}
\usepackage{graphicx}
\usepackage{latexsym}
\usepackage{amsxtra}
\usepackage{amsmath}
\usepackage{amsfonts}
\usepackage{amssymb}

\newtheorem{theorem}{Theorem}[section]

\newtheorem{proposition}[theorem]{Proposition}

\usepackage[figuresright]{rotating}

\title{A Fuss-type family of positive definite sequences}
\author{Wojciech M{\l}otkowski}
\author{Karol A. Penson}
\thanks{
W.~M. is supported by the Polish
National Science Center grant No. 2012/05/B/ST1/00626.}
\address{Instytut Matematyczny,
Uniwersytet Wroc{\l}awski,
Plac~Grunwaldzki~2/4,
50-384 Wroc{\l}aw, Poland}
\email{mlotkow@math.uni.wroc.pl}
\address{Sorbonne Universit\'{e}s, Universit\'{e} Pierre et Marie Curie,
Laboratoire de Physique Th\'{e}orique de la Mati\`{e}re
Condens\'{e}e (LPTMC), CNRS UMR
7600, Tour 13 - 5i\`{e}me \'{e}t., Bo\^{i}te Courrier 121, 4 place
Jussieu, F 75252 Paris Cedex 05, France}
\email{penson@lptl.jussieu.fr}

\begin{document}

\begin{abstract}
We study a two-parameter family $a_{n}(p,t)$
of deformations of the Fuss numbers.
We show a sufficient condition for positive definiteness of $a_n(p,t)$
and prove that some of the corresponding probability measures
are infinitely divisible with respect to the additive free convolution.
\end{abstract}

\maketitle

\section{Introduction}

The aim of the paper is to study a two-parameter family
of sequences $a_{n}(p,t)$, $p,t\in\mathbb{R}$, defined by (\ref{oursequence1}),
which can be regarded as deformation of the Fuss numbers.
Assuming that $p\ge 0$ we prove that the sequence $a_{n}(p,t)$ is positive definite if and only if
$p\ge1$ and $g(p)\le t\le 2p/(p+1)$, where $g(p)$ is defined by (\ref{functiong}).
We conjecture that the assumption that $p\ge0$ is redundant.

The case $t=2p/(p+1)$ is particularly interesting by connections with the work \cite{bousquetschaeffer} of
M.~Bousquet-M\'{e}lou and G.~Schaeffer.
They introduced the notion of constellation
as a tool for studying factorization problems in the symmetric groups.
For $p\ge2$ a \textit{$p$-constellation} is a $2$-cell decomposition
of the oriented sphere into vertices, edges and faces,
with faces colored black and white in such a way that:

\begin{itemize}

\item all faces adjacent to a given white face are black and vice versa,

\item the degree of any black face is $p$,

\item the degree of any white face is a multiple of $p$.

\end{itemize}
A constellation is called \textit{rooted} if one of the edges is distinguished.

The number of rooted $p$-constellations formed of $n$ polygons,
counted up to isomorphism, is given by
\begin{equation}\label{constellations}
C_p(n):=\binom{np}{n}\frac{(p+1)p^{n-1}}{(np-n+1)(np-n+2)},
\end{equation}
$p\ge2$, $n\ge1$,
see Corollary~2.4 in~\cite{bousquetschaeffer}.
Some of these sequences appear in the On-line Encyclopedia of Integer Sequences (OEIS) \cite{oeis}, namely: $C_2=A000257$, $C_3=A069726$, $C_4=A090374$.

We will prove that the probability distribution $\eta(p,t)$
corresponding to positive definite sequence $a_n(p,t)$ is absolutely
continuous, except for $\eta(1,1)=\delta_1$,
and the support of $\eta(p,t)$ is $[0,p^p (p-1)^{1-p}]$.
The density function will be denoted $f_{p,t}(x)$.
For $p=2$ and $p=3$ we compute the $R$-transform of $\eta(p,t)$.
We prove that $\eta(2,p)$ (resp. $\eta(3,t)$) is infinitely divisible
with respect to the additive free convolution if and only if
$1\le t\le 4/3$ (resp. $1/2\le t\le 3/2$).

Finally, let us record some other sequences from OEIS which are related to this work:
$A005807$: $2a_n(2,1/2)$ (sums of adjacent Catalan numbers),
$A007226$: $2a_n(3,1/2)$ (studied in~\cite{mlope2013}),
$A007054$: $3a_n(2,4/3)$ (super ballot numbers),
$A038629$: $3a_n(2,2/3)$,
$A000139$: $2a_{n}(3,3/2)$,
$A197271$: $5 a_{n}(4,8/5)$,
$A197272$: $3 a_{n}(5,5/3)$.
In Section~\ref{sectionfree}  we also encounter sequences $A022558$ and $A220910$.

\section{Fuss numbers}\label{sectionfuss}
The \textit{Fuss-Catalan numbers} $\binom{np+1}{n}\frac{1}{np+1}$ have several combinatorial applications,
see \cite{gkp,edelman,armstrong2009,schuetzwhieldon,bousquetschaeffer,przytyckisikora2000}. They count for example:
\begin{enumerate}

\item the number of ways of subdividing a convex polygon, with $n(p-1)+2$ vertices,
into $n$ disjoint $p+1$-gons by means of nonintersecting diagonals,

\item the number of sequences $(a_1,a_2,\ldots,a_{np})$, where $a_i\in\{1,1-p\}$,
with all partial sums $a_1+\ldots+a_k$ nonnegative and with $a_1+\ldots+a_{np}=0$,

\item the number of noncrossing partitions $\pi$ of $\{1,2,\ldots,n(p-1)\}$,
such that $p-1$ divides the cardinality of every block of $\pi$,

\item the number of $p$-cacti formed of $n$ polygons,
see \cite{bousquetschaeffer}. 
\end{enumerate}

The generating function:
\begin{equation}
\mathcal{B}_{p}(z):=\sum_{n=0}^{\infty}\binom{np+1}{n}\frac{z^n}{np+1}
\end{equation}
satisfies
\begin{equation}\label{fussrelation}
\mathcal{B}_p(z)=1+z\mathcal{B}_{p}(z)^p.
\end{equation}
Recall also the Lambert's formula for the Taylor expansion of the powers of $\mathcal{B}_{p}(z)$:
\begin{equation}
\mathcal{B}_{p}(z)^r=\sum_{n=0}^{\infty}\binom{np+r}{n}\frac{r z^n}{np+r}.
\end{equation}
These formulas remain true for $p,r\in\mathbb{R}$ and
the coefficients $\binom{np+r}{n}\frac{r}{np+r}$
(understood to be $1$ for $n=0$
and $\frac{r}{n!}\prod_{i=1}^{n-1}(np+r-i)$ for $n\ge1$)
are called \textit{two-parameter Fuss numbers} or \textit{Raney numbers},
see \cite{gkp,mlotkowski2010,pezy,liupego2014,forresterliu2014}.

In some cases the function $\mathcal{B}_p$ can be written explicitly, for example
\begin{align*}
\mathcal{B}_2(z)&=\frac{2}{1+\sqrt{1-4z}}
=\frac{1-\sqrt{1-4z}}{2z},\\
\mathcal{B}_3(z)
&=\frac{3}{3-4\sin^2\alpha},\\
\mathcal{B}_{3/2}(z)
&=\frac{3}{\left(\sqrt{3}\cos\beta-\sin\beta\right)^2},
\end{align*}
where $\alpha=\frac{1}{3}\arcsin\left(\sqrt{27z/4}\right)$,
$\beta=\frac{1}{3}\arcsin\left(3z\sqrt{3}/2\right)$, see~\cite{mlotkowskipensonbinomial}.

Fuss numbers also have applications in free probability and in the theory of random matrices,
as moments of the multiplicative free powers of the Marchenko-Pastur distribution
\cite{alexeev2010,bbcc,mlotkowski2010,mlonopezy,neuschel}.
This implies that for $p\ge1$ the sequence $\binom{np+1}{n}\frac{1}{np+1}$ is
positive definite. More generally, the sequence $\binom{np+r}{n}\frac{r}{np+r}$
is positive definite if and only if either $p\ge0$, $0\le r\le p$, or $p\le0$, $p-1\le r\le0$ or $r=0$,
see \cite{mlotkowski2010,mlotkowskipensonbinomial,liupego2014,forresterliu2014}.
The case $r=0$ is trivial, as it gives the sequence $1,0,0,0,\ldots$,
moments of $\delta_0$. The distributions corresponding to the second case,  $p\le0$, $p-1\le r\le0$,
are just reflections of those corresponding to $p\ge0$, $0\le r\le p$.
It is a consequence of the identity
\begin{equation}
\binom{np+r}{n}\frac{r(-1)^n}{np+r}=\binom{n(1-p)-r}{n}\frac{-r}{n(1-p)-r}.
\end{equation}

For $p>1$, $r>0$ we have the following integral representation:
\[
\binom{np+r}{n}\frac{r}{np+r}=\int_{0}^{c(p)}x^n W_{p,r}(x)\,dx,
\]
where where $c(p):=p^p(p-1)^{1-p}$, and $W_{p,r}$ can be described as:
\begin{equation}\label{forrestera}
W_{p,r}(x)=\frac{\left(\sin(p-1)\phi\right)^{p-r-1}\sin\phi \sin r\phi}{\pi\left(\sin p\phi\right)^{p-r}},
\end{equation}
where
\begin{equation}\label{forresterb}
x=\rho(\phi)=\frac{\left(\sin p\phi\right)^p}{\sin\phi\left(\sin(p-1)\phi\right)^{p-1}},\quad 0<\phi<\pi/p.
\end{equation}
This function is nonnegative if and only if $r\le p$, see \cite{haaagerupmoller2014,neuschel,forresterliu2014}.

If $p=k/l$ is a rational number, $1\le l<k$, then
$W_{p,r}$ can be expressed in terms of the Meijer $G$-function (see \cite{pezy,mlopezy2013}):
\begin{equation}\label{meijerth1}
W_{p,r}(x)=\frac{rp^{r}}{x(p-1)^{r+1/2}\sqrt{2k\pi}}\,
G^{k,0}_{k,k}\!\left(\frac{x^{l}}{c(p)^l}\left|\!\!
\begin{array}{ccc}
\alpha_1,\!\!&\!\!\ldots,\!\!&\!\!\alpha_k\\
\beta_1,\!\!&\!\!\ldots,\!\!&\!\!\beta_k
\end{array}
\!\right.\right),
\end{equation}
$x\in(0,c(p))$ and the parameters $\alpha_j,\beta_j$ are given by:
\begin{align}
\alpha_j&=\left\{\begin{array}{ll}
\cfrac{j}{l}&\mbox{if $1\le j\le l$,}\label{convalpha1}\\
\cfrac{r+j-l}{k-l}&\mbox{if $l+1\le j\le k$,}
\end{array}\right.\\
\beta_j&=\frac{r+j-1}{k},\quad\qquad{1\le j\le k}.\label{convbeta1}
\end{align}

\textbf{Examples:} Let us record formulas for the functions $W_{p,r}$ for $p=2,3,3/2$ and $r=1,2$.
In these cases $W_{p,r}$ can be expressed as an elementary function,
see \cite{pensonsolomon,pezy,mlopezy2013}.
\begin{align}
W_{2,1}(x)&=\frac{1}{2\pi}\sqrt{\frac{4-x}{x}},\\
W_{2,2}(x)&=\frac{1}{2\pi}\sqrt{x(4-x)},\label{eta20}
\end{align}
where $x\in(0,4)$. $W_{2,1}$ is the density of the Marchenko-Pastur distribution
and $W_{2,2}$ is the Wigner's semicircle law translated by~2.

\begin{align}
W_{3,1}(x)&=\frac{3\left(1+\sqrt{1-4x/27}\right)^{2/3}-(4x)^{1/3}}
{3^{1/2}\pi (4x)^{2/3}\left(1+\sqrt{1-4x/27}\right)^{1/3}},\\
W_{3,2}(x)&=\frac{9\left(1+\sqrt{1-4x/27}\right)^{4/3}-(4x)^{2/3}}
{2\pi 3^{3/2} (4x)^{1/3}\left(1+\sqrt{1-4x/27}\right)^{2/3}},
\end{align}
where $x\in(0,27/4)$.

\begin{align}
W_{3/2,1}(x)=3^{1/2}&\frac{\left(1+\sqrt{1-4x^2/27}\right)^{1/3}
-\left(1-\sqrt{1-4x^2/27}\right)^{1/3}}
{2(2x)^{1/3}\pi}\\
+3^{1/2}(2x)^{1/3}&\frac{\left(1+\sqrt{1-4x^2/27}\right)^{2/3}
-\left(1-\sqrt{1-4x^2/27}\right)^{2/3}}
{4\pi},\nonumber
\end{align}
\begin{align}
W_{3/2,2}(x)=\frac{3^{1/2}(2x)^{5/3}}{8\pi}
&\left(\left(1+\sqrt{1-4x^2/27}\right)^{1/3}-\left(1-\sqrt{1-4x^2/27}\right)^{1/3}\right)\\
+\frac{3^{1/2}(2x)^{1/3}(x^2-1)}{4\pi}
&\left(\left(1+\sqrt{1-4x^2/27}\right)^{2/3}-\left(1-\sqrt{1-4x^2/27}\right)^{2/3}\right),\nonumber
\end{align}
where $x\in(0,3\sqrt{3}/2)$. The function $W_{3/2,2}(x)$ is not nonnegative on its domain.

\section{A family of sequences}\label{sectionsequence}
For $p,t\in\mathbb{R}$ define sequence $a_n(p,t)$ as an affine combination
of $\binom{np+1}{n}\frac{1}{np+1}$ and $\binom{np+2}{n}\frac{2}{np+2}$:
\begin{align}
a_{n}(p,t):=
&\binom{np+1}{n}\frac{t}{np+1}+\binom{np+2}{n}\frac{2(1-t)}{np+2}\label{oursequence1}\\
=&\binom{np}{n}\frac{n(2p-t-pt)+2}{(np-n+1)(np-n+2)},
\end{align}
in particular $a_0(p,t)=1$.

The generating function is
\begin{equation}
t\mathcal{B}_{p}(z)+(1-t)\mathcal{B}_{p}(z)^2=\sum_{n=0}^{\infty}a_{n}(p,t)z^n.
\end{equation}
For example:
\begin{align*}
t\mathcal{B}_2(z)+(1-t)\mathcal{B}_2(z)^2
&=\frac{1-t+3tz-2z-(1-t+tz)\sqrt{1-4z}}{2z^2},\\
t\mathcal{B}_3(z)+(1-t)\mathcal{B}_3(z)^2
&=\frac{9-12t\sin^2\alpha}{\left(3-4\sin^2\alpha\right)^2},\\
t\mathcal{B}_{3/2}(z)+(1-t)\mathcal{B}_{3/2}(z)^2
&=\frac{9-6t\sin^2\beta+6t\sqrt{3}\sin\beta\cos\beta}{\left(\sqrt{3}\cos\beta-\sin\beta\right)^4}
\end{align*}
where $\alpha=\frac{1}{3}\arcsin\left(\sqrt{27z/4}\right)$,
$\beta=\frac{1}{3}\arcsin\left(3z\sqrt{3}/2\right)$.

We are going to study positive definiteness of $a_n(p,t)$.
First we observe
\begin{proposition}\label{2propositionnecessary}
If the sequence $a_{n}(p,t)$ is positive definite then
\begin{equation}\label{2necessarycondition}
2p-pt-t^2+3t-3\ge0.
\end{equation}
In particular $t\ne2$ and either $p\le-3$ or $p\ge1$.
\end{proposition}

\begin{proof}
The left hand side is just $a_2(p,t)-a_1(p,t)^2$.
\end{proof}

\textbf{Examples.}

\textbf{1.}
For $p=1$ we have $a_n(1,t)=1+n-nt$.
Since $a_2(1,t)-a_1(1,t)^2=-(t-1)^2$, the sequence $a_n(1,t)$
is positive definite if and only if $t=1$.
Note that $a_n(1,1)=1$ is the moment sequence of
the one-point measure $\delta_1$.

\textbf{2.} For $t=2/(p+1)$ we get
\[
a_{n}\left(p,2/(p+1)\right)=\binom{np}{n}\frac{2}{np-n+2}.
\]
If $p>1$ then this is product of two positive definite sequences:
$\binom{np}{n}$ (see \cite{mlotkowskipensonbinomial,simon2014}) and $2/(np-n+2)$.

\textbf{3.} Similarly, for $p>1$, $t=2p/(p+1)$ the sequence
\[
a_{n}\left(p,2p/(p+1)\right)=\binom{np}{n}\frac{2}{(np-n+1)(np-n+2)}.
\]
is positive definite.
Note that from (\ref{constellations}) we have
\begin{equation}
C_p(n)=\frac{(p+1)p^n}{2p}a_n\left(p,\frac{2p}{p+1}\right),
\end{equation}
so for $p\ge1$ the sequence $C_{p}(n)$ is positive definite.

The sequence $a_{n}(p,t)$ is an
affine combination of two sequences:
$\binom{np+1}{n}\frac{1}{np+1}$ and
$\binom{np+2}{n}\frac{2}{np+2}$.
The former is positive definite
for $p\ge1$ and the latter for $p\ge2$.
This implies, that $a_{n}(p,t)$ is positive definite for $p\ge2$, $0\le t\le 1$.
We are going to prove something stronger.
Note that if $t_1\le t_2\le t_3$ and the sequences $a_{n}(p,t_1)$, $a_n(p,t_3)$
are positive definite then so is $a_n(p,t_2)$ as their convex combination.

If we assume that $p>1$ then
\[
a_{n}(p,t)=\int_{0}^{c(p)}x^n f_{p,t}(x)\,dx,
\]
where
\[
f_{p,t}(x)=t W_{p,1}(x)+(1-t)W_{p,2}(x).
\]
Then the positive definiteness of $a_n(p,t)$ is equivalent to the fact
that $f_{p,t}$ is nonnegative on $(0,c(p))$.
For example the function
\begin{equation}\label{2densityfor2}
f_{2,t}(x)=\frac{t+x-tx}{2\pi}\sqrt{\frac{4-x}{x}}
\end{equation}
is nonnegative on $(0,4)$ if and only if $0\le t\le 4/3$.

\begin{figure}
\caption{The density function $f_{3/2,1/5}(x)$}
\centering
\includegraphics[width=0.7\textwidth]{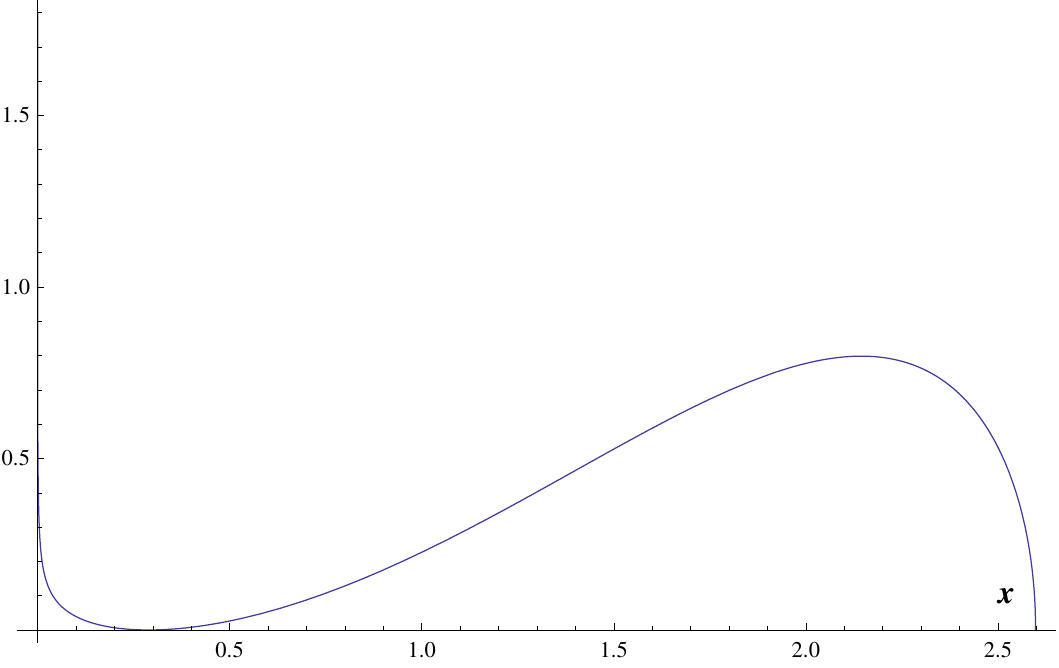}
\label{case3215}
\end{figure}

By (\ref{forrestera}) we can write
\begin{equation}
f_{p,t}(x)=\frac{\sin^2\phi\left(\sin(p-1)\phi\right)^{p-3}\big[t\sin(p-1)\phi+2(1-t)\sin p\phi\cos\phi\big]}
{\pi\left(\sin p\phi\right)^{p-1}}
\end{equation}
for $x$ as in (\ref{forresterb}).
Define
\begin{align}
\Psi_{p,t}(\phi)&=t\sin(1-1/p)\phi+2(1-t)\sin\phi\cos\phi/p\label{psi}\\
&=(2-t)\sin\phi\cos\phi/p-t\cos\phi\sin\phi/p\nonumber\\
&=(1-t)\sin(1+1/p)\phi+\sin(1-1/p)\phi.\nonumber
\end{align}
Then the sequence $a_{n}(p,t)$ is positive definite if and only if
$\Psi_{p,t}(\phi)\ge0$ for $\phi\in[0,\pi]$.
For $p\ge1$ put
\begin{equation}\label{functiong}
g(p):=\min\{t\in \mathbb{R}:\Psi_{p,t}(\phi)\ge0\hbox{ for all }0<\phi<\pi\}.
\end{equation}
Since $\Psi_{p,t}(\pi)=t\sin(\pi/p)$ and $\Psi_{p,1}(\phi)=\sin(1-1/p)\phi$,
we have $0\le g(p)\le1$ for all $p\ge1$.

\begin{proposition}
The function $g$ is continuous on $[1,\infty)$,
$g(1)=1$, $g(p)=0$ for $p\ge2$ and is strictly decreasing on $[1,2]$.
In particular $g(3/2)=1/5$.
\end{proposition}

\begin{proof}
For $p=1$ we have $\Psi_{1,t}(\phi)=(1-t)\sin 2\phi$, which implies $g(1)=1$.
If $p\ge2$ then $\Psi_{p,0}(\phi)=2\sin\phi\cos\phi/p$ is nonnegative for $\phi\in[0,\pi]$,
which yields $g(p)=0$.

Now observe, that for fixed $t,\phi$, with $0\le t\le1$,
$0<\phi\le\pi$, the function $p\mapsto\Psi_{p,t}(\phi)$ is strictly increasing on $[1,2]$.
Indeed, we can write
\[
\Psi_{p,t}(\phi)=2(1-t)\sin\phi\cos\phi/p+t\sin(\phi-\phi/p)
\]
and if $0<\phi\le\pi$ then both the summands are increasing with $p\in[1,2]$.
This implies, that $g(p)$ is strictly decreasing on $[1,2]$.

To prove continuity of $g$ assume that $1\le p_1<p_2\le2$ and put $t_1:=g(p_1)$, $t_2:=g(p_2)$.
Then $t_1>t_2$, $\Psi_{p_1,t_1}(\phi)\ge0$ for all $\phi\in[0,\pi]$ and there is $\phi_1$, with $p_1\pi/(1+p_1)<\phi_1<\pi$,
such that $\Psi_{p_1,t_1}(\phi_1)=0$.
Then we have that $\Psi_{p_2,t_1}(\phi)>0$ for all $\phi\in(0,\pi]$.
From the third expression in (\ref{psi}) we have that
\[
-c_1:=\sin(1+1/p_1)\phi_1<0.
\]
If we assume that $(p_2-p_1)\phi_1<c_1/2$ then we have
\[
\left|\sin(1+1/p_1)\phi_1-\sin(1+1/p_2)\phi_1\right|\le\left(1/p_1-1/p_2\right)\phi_1<c_1/2
\]
and, consequently, $\sin(1+1/p_2)\phi_1<-c_1/2$.

If we take $t$, with $0\le t<t_1$, then
\[
\Psi_{p_2,t}(\phi_1)=\Psi_{p_2,t}(\phi_1)-\Psi_{p_1,t_1}(\phi_1)
\]
\[
=(1-t_1)\big(\sin(1+1/p_2)\phi_1-\sin(1+1/p_1)\phi_1\big)
+\big(\sin(1-1/p_2)\phi_1-\sin(1-1/p_1)\phi_1\big)
\]
\[
+(t_1-t)\sin(1+1/p_2)\phi_1
\le(2-t_1)(p_2-p_1)\phi_1-(t_1-t)c_1/2.
\]
Hence, if
\[
(2-t_1)(p_2-p_1)\phi_1<(t_1-t)c_1/2
\]
then $\Psi_{p_2,t}(\phi_1)<0$.
This implies that
\[
g(p_1)-g(p_2)=
t_1-t_2\le2(2-t_1)(p_2-p_1)\phi_1/c_1.
\]
and proves continuity of $g$.

For $p=3/2$ we can write
\[
\Psi_{3/2,t}(\phi)
=\frac{\sin\phi/3}{4}\left[(1-t)\left(5-8\sin^2\phi/3\right)^2+5t-1\right].
\]
Note that $\sqrt{5/8}<\sqrt{3}/2=\sin\pi/3$,
so, assuming that $0\le t\le 1$, $\Psi_{3/2,t}$ attains its minimum
on $[0,\pi]$ at $\phi=3\arcsin\sqrt{5/8}$. This yields $g(3/2)=1/5$.
\end{proof}

Now we are able to describe the domain of positive definiteness
of the sequence $a_n(p,t)$, see Fig~\ref{domainofpositivedef}.
The density function for the particular case
$p=3/2$, $t=1/5$ is illustrated in Fig.~\ref{case3215}.

\begin{theorem}
Suppose that $p\ge0$. Then the sequence $a_n(p,t)$ is positive definite
if and only if $p\ge1$ and
\begin{equation}
g(p)\le t\le \frac{2p}{1+p}.
\end{equation}
\end{theorem}

\begin{figure}
\caption{Domain of positive definiteness of the sequence $a_n(p,t)$}
\centering
\includegraphics[width=0.7\textwidth]{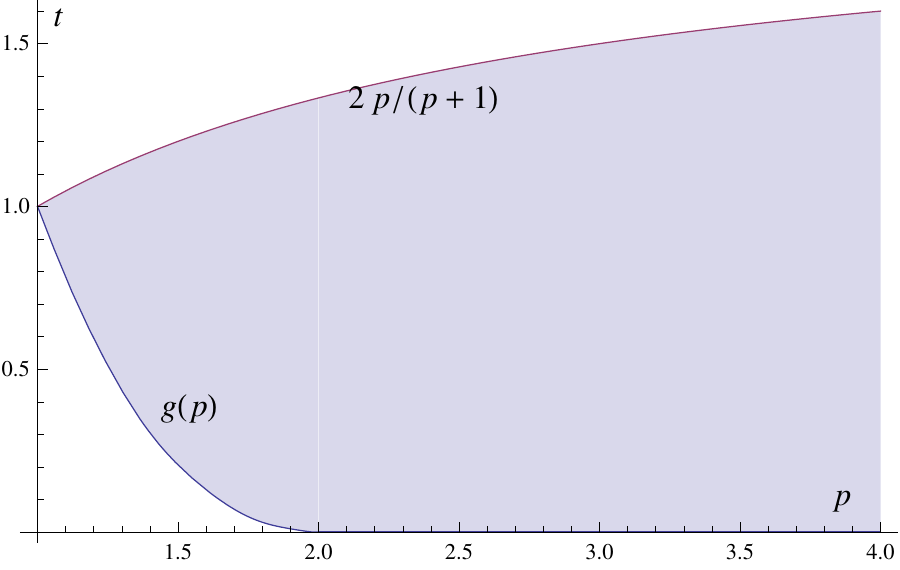}
\label{domainofpositivedef}
\end{figure}

\begin{proof}
Fix $p\ge1$. By the definition of $g(p)$ the sequence $a_n(p,t)$ is positive definite for $t=g(p)$ and
not positive definite for $t<g(p)$.

We have already observed, that for $p\ge1$ the sequence $a_{n}(p,2p/(p+1))$
is positive definite. If $t>2p/(p+1)$ then $n(2p - t - p t)+2<0$ and consequently $a_n(p,t)<0$
for all $n$ sufficiently large. Alternatively, we have $\Psi_{p,t}'(0)=2p-pt-t<0$
in this case, which implies $\Psi_{p,t}(x)<0$ for some $x\in(0,\pi/p)$.
\end{proof}

\section{Free transforms}\label{sectionfree}

Throughout this section we assume that $p\ge1$ and
the sequence $a_{n}(p,t)$ is positive definite, i.e. $g(p)\le t\le 2p/(p+1)$.
Denote by $\eta(p,t)$ the corresponding distribution,
i.e. $\eta(1,1)=\delta_1$ and $\eta(p,t)=f_{p,t}(x)\,dx$ on $[0,p^p (p-1)^{1-p}]$ for $p>1$.
We are going to study relations of these measures with free probability.

Recall that for a compactly supported probability measure $\mu$ on $\mathbb{R}$,
with the moment generating function
\begin{equation}
M_{\mu}(z):=\sum_{n=0}^{\infty} z^n\int_{\mathbb{R}} x^n\,d\mu(x)
=\int_{\mathbb{R}}\frac{1}{1-xz}\,d\mu(x),
\end{equation}
the $S$- and $R$-transforms are defined by
\begin{equation}\label{3stransform}
M_{\mu}\left(\frac{z}{1+z}S_{\mu}(z)\right)=1+z,
\end{equation}
\begin{equation}
1+R_{\mu}\left(z M_{\mu}(z)\right)=M_{\mu}(z).
\end{equation}
Moreover, we have relation
\begin{equation}\label{3rstransformrelation}
R_{\mu}\left(z S_{\mu}(z)\right)=z.
\end{equation}
The coefficients $r_n(\mu)$ in the Taylor expansion
$R_{\mu}(z)=\sum_{n=1}^{\infty} r_n(\mu)z^n$ are called \textit{free cumulants} of $\mu$.
It is known that $\mu$ is infinitely divisible with respect to the additive free
convolution if and only if the sequence $\{r_{n+2}(\mu)\}_{n=0}^{\infty}$
is positive definite, see~\cite{vdn,nicaspeicherlectures}.

For the distributions $\eta(p,t)$ we have
\[
M_{\eta(p,t)}(z):=\sum_{n=0}^{\infty}a_{n}(p,t)z^n=t\mathcal{B}_{p}(z)+(1-t)\mathcal{B}_{p}(z)^2.
\]

Now we are going to compute the $S$-transform of $\eta(p,t)$.

\begin{proposition} For $p>1$, $g(p)\le t\le 2p/(p+1)$ we have
\begin{equation}\label{sfreetransform}
S_{\eta(p,t)}(w)=(2+2w)^{1-p}\frac{\left(\sqrt{(2-t)^2 + 4 (1 - t) w}+t\right)^p}{\sqrt{(2-t)^2 + 4 (1 - t) w}+2-t}.
\end{equation}
\end{proposition}

\begin{proof}
From (\ref{fussrelation}) we can derive relation
\[\mathcal{B}_{p}\left(z(1+z)^{-p}\right)=1+z,\]
see~\cite{mlotkowski2010}.
Therefore
\[
M_{\eta(p,t)}\left(z(1+z)^{-p}\right)=t (1 + z) + (1 - t) (1 + z)^2.
\]
If we substitute
\[
t (1 + z) + (1 - t) (1 + z)^2 =1 + w
\]
then
\[
z = \frac{\sqrt{(2-t)^2 + 4 (1 - t) w}-2+t}{2 (1 - t)}
= \frac{2w}{\sqrt{(2-t)^2 + 4 (1 - t) w}+2-t}
\]
and
\[
1+z = \frac{\sqrt{(2-t)^2 + 4 (1 - t) w}-t}{2 (1 - t)}
=\frac{2(1+w)}{\sqrt{(2-t)^2 + 4 (1 - t) w}+t},
\]
which combining with (\ref{3stransform}) yields~(\ref{sfreetransform}).
\end{proof}

Now we are going to compute $R$-transform of $\eta(p,t)$ for $p=2$ and $p=3$.
We will denote $r_{n}(p,t):=r_{n}(\eta(p,t))$.

\subsection{The case $p=2$}

The density function $f_{2,t}$ is given by (\ref{2densityfor2}), $0\le t\le 4/3$.
From (\ref{sfreetransform}) we can compute the $R$-transform for $p=2$:

\begin{proposition}
$R_{\eta(2,1)}=z/(1-z)$ and for $t\ne1$
\[
R_{\eta(2,t)}(z)=
\frac{1 - t - 2 z + 3 t z - z^2 + (t-1-z)\sqrt{1+z(2-4t)+z^2}}{2(t-1)}.
\]
Moreover, $\eta(2,t)$ is infinitely divisible with respect to the additive free convolution
if and only if either $t=0$ or $1\le t\le4/3$.
\end{proposition}

\begin{proof}
First we find $R_{\eta(2,t)}(z)$ by solving equation $S_{\eta(2,t)}\left(R_{\eta(2,t)}(z)\right) R_{\eta(2,t)}(z)=z$,
equivalent with (\ref{3rstransformrelation}), with the condition $R_{\eta(2,t)}(0)=0$.
In particular $R_{2,0}=2z+z^2$, which implies that $\eta(2,0)$ is infinitely divisible
with respect to the additive free convolution.

Now we can find:
\begin{align*}
r_1(2,t)&=2 - t,\\
r_2(2,t)&=1 + t - t^2,\\
r_3(2,t)&=3 t^2 - 2 t^3,\\
r_4(2,t)&=-4 t^2 + 10 t^3 - 5 t^4.
\end{align*}
Since
\[
r_2(2,t)r_4(2,t)-r_3(2,t)^2=t^2(t-1) (t-2)(t^2-2),
\]
for $0<t<1$ the distribution $\eta(2,t)$ is not infinitely divisible
with respect to the additive free convolution.

For $t\ne 1$ we have
\[
1+R_{\eta(2,t)}(z)=\frac{t - 1 - 2 z + 3 t z - z^2 + (t-1-z)\sqrt{1+z(2-4t)+z^2}}{2(t-1)}
\]
and $1+R_{\eta(2,1)}(z)=1/(1-z)$.
Then for $1<t\le3/2$ the function
\[
\frac{1+R_{\eta(2,t)}(1/z)}{z}
=\frac{(t-1)z^2 - 2 z + 3 t z -1 + \big(z(t-1)-1\big)\sqrt{1+z(2-4t)+z^2}}{2(t-1)z^3}
\]
is the Cauchy transform of the probability distribution
\[
\frac{(1-tx+x)\sqrt{4t(t-1)-(x-2t+1)^2}}{2\pi(t-1)x^3}\,dx,
\]
on the interval
\[
x\in\left[2t-1-2\sqrt{t^2-t},2t-1+2\sqrt{t^2-t}\right].
\]
Therefore for $1< t\le 4/3$
\begin{equation}\label{3integral222}
r_n(2,t)=\int_{2t-1-2\sqrt{t^2-t}}^{2t-1+2\sqrt{t^2-t}}
x^n\frac{(1-tx+x)\sqrt{4t(t-1)-(x-2t+1)^2}}{2\pi(t-1)x^3}\,dx,
\end{equation}
which proves that
the sequence $\{r_{n+2}(2,t)\}_{n=0}^{\infty}$
is positive definite.
\end{proof}

\textbf{Remark.}
Note, that for $\eta(2,0)$ the cumulant sequence is $(2,1,0,0,\ldots)$,
so the sequence $\{r_{n+2}(2,0)\}_{n=0}^{\infty}=(1,0,0,\ldots)$ is positive definite.
Actually, $\eta(2,0)$,
given by (\ref{eta20}),
is a translation of the Wigner semicircle distribution
$\frac{1}{2\pi}\sqrt{4-x^2}\,dx$, $x\in[-2,2]$. The free additive infinite divisibility
of $\eta(2,0)$ was overlooked in~\cite{mlotkowskipensonbinomial},
Corollary~7.1, where $\eta(2,0)$ was denoted $\mu(2,2)$.

\textbf{Example 1.}
Define a sequence $a_n$ by $a_{0}:=1$ and $a_n:=3^{n}\cdot r_n(2,4/3)$ for $n\ge1$:
\[
1, 2, 5, 16, 64, 304, 1632, 9552, 59520, 388720, 2632864,\ldots.
\]
Applying (\ref{3integral222}) for $t=4/3$ we obtain
\begin{equation}\label{3example2aint}
a_n=\int_{1}^{9} x^n\frac{\sqrt{(x-1)(9-x)^3}}{2\pi x^3}\,dx.
\end{equation}
Its generating function is
\begin{equation}\label{3example2agenf}
\sum_{n=0}^{\infty}a_n z^n=
1+R_{\eta(2,4/3)}(3z)=\frac{1+18z-27z^2+\sqrt{(1-z)(1-9z)^3}}{2}.
\end{equation}

\textbf{Example 2.}
Now let us consider the binomial transform of $a_n$:
\[
b_n:=\sum_{k=0}^{n}(-1)^{n-k}\binom{n}{k}a_k.
\]
The corresponding density function is that of the sequence $a_n$
translated by $-1$, so
\begin{equation}
b_n=\int_{0}^{8} x^n\frac{\sqrt{x(8-x)^3}}{2\pi (x+1)^3}\,dx.
\end{equation}
For the generating function we have
\[
\sum_{n=0}^{\infty}b_n z^n
=\sum_{k=0}^{\infty} a_k(-1)^k\sum_{n=k}^{\infty}\binom{n}{k}(-z)^n
\]
\[
=\sum_{k=0}^{\infty}a_k \frac{z^k}{(1+z)^{k+1}}
=\frac{1}{1+z}\left(1+R_{\eta(2,4/3)}\left({3z}/{(1+z)}\right)\right),
\]
so from (\ref{3example2agenf})
\begin{equation}\label{3avoiding2genf}
\sum_{n=0}^{\infty}b_n z^n=\frac{1+20z-8z^2+\sqrt{(1-8z)^3}}{2(1+z)^3}.
\end{equation}
This proves that $b_n$ coincides with $A022558$ in OEIS:
\[
1, 1, 2, 6, 23, 103, 512, 2740, 15485, 91245, 555662,\ldots,
\]
which counts the permutations of length $n$ which
avoid the pattern $1342$, see Theorem~2 in~\cite{bona1997}.

\subsection{The case $p=3$}
\begin{proposition}
\begin{equation}
R_{\eta(3,t)}(z)=\frac{
z(4 - 7 t + 4 t^2 - 2 z)-(t-1)^2 + (1-2t+t^2-tz)\sqrt{1-4tz}}{2 (t + z-1)^2}
\end{equation}
and the distribution $\eta(3,t)$ is infinitely
divisible with respect to the additive free convolution if and only if $1/2\le t\le 3/2$.
\end{proposition}

\begin{proof}
The proof is similar as for $p=2$.
First we find $R_{\eta(3,t)}$ by solving the equation
\[
S_{\eta(3,t)}\left(R_{\eta(3,t)}(z)\right) R_{\eta(3,t)}(z)=z,
\]
with the condition that $R_{\eta(3,t)}(0)=0$.
Then we find out that
\[
1+R_{\eta(3,t)}(z)=\frac{
(t-1)^2 +tz(4t-3) + (1-2t+t^2-tz)\sqrt{1-4tz}}{2 (t + z-1)^2}
\]
is the moment generating function for
the density
\begin{equation}\label{3cumulantsmoments3t}
\frac{\left(t-x(t-1)^2\right)\sqrt{4t-x}}{2\pi (tx-x+1)^2\sqrt{x}},\qquad x\in[0,4t],
\end{equation}
which is positive provided $1/2\le t\le 3/2$.
\end{proof}

\textbf{Example.}
The sequence $a_n={A220910}(n)$:
\[
1, 1, 3, 14, 83, 570, 4318, 35068, 299907, 2668994, 24513578,\ldots
\]
counts matchings avoiding the pattern $231$,
see \cite{bloomelizalde} for details.
Its generating function equals
\begin{equation}
M(z)=\sum_{n=0}^{\infty}a_n z^n=\frac{1+36z+\sqrt{(1-12z)^3}}{2(1+4z)^2}=1+R_{\eta(3,3/2)}(2z),
\end{equation}
so we have $a_n=2^n\cdot r_{n}(3,3/2)$ for $n\ge1$.
Therefore these numbers can be represented as moments:
\begin{equation}\label{3moments}
a_n=\int_{0}^{12}x^n\frac{\sqrt{(12-x)^3}}{2\pi(x+4)^2\sqrt{x}}\,dx.
\end{equation}

Now we are going to prove a recurrence relation,
which was was conjectured by R.~J.~Mathar
(see OEIS, entry {A220910}, Aug. 04 2013).

\begin{proposition}
For $n\ge2$ we have
\begin{equation}\label{3recurrence}
n a_{n} =(8n-34)a_{n-1} +24(2n-3)a_{n-2}.
\end{equation}
\end{proposition}

\begin{proof}
One can check that the generating function satisfies differential equation:
\[
(1-8z-48z^2)M'(z)+(26-24z)M(z)=27.
\]
The coefficient at $z^{n-1}$ on the left hand side is equal to
\[
n a_{n } - 8 (n-1) a_{n-1} - 48 (n - 2)a_{n - 2} +
 26 a_{n-1} - 24a_{n - 2}
\]
for $n\ge2$, which gives (\ref{3recurrence}).
\end{proof}

Now we will provide two formulas for~$a_n={A220910}(n)$.

\begin{proposition}
\begin{equation}\label{3avoiding3suma}
a_n=\frac{1-8n}{2}(-4)^{n}
+\binom{2n}{n}\sum_{k=0}^{n}\frac{3^{n+1}(k+1)\prod_{i=0}^{k-1}{(n-i)}}{8(-3)^k\prod_{i=0}^{k+1}{(n-i-1/2)}}
\end{equation}
\begin{equation}\label{3avoiding3sumb}
=\frac{(-4)^n(1-8n)}{16}
\left[
8-\sum_{k=0}^{n+1}\frac{(-3)^k}{k!}\prod_{i=0}^{k-1}(i-3/2)
\right]
+\binom{2n}{n}\frac{3^{n+3}}{32(n+1)}.
\end{equation}
\end{proposition}

\begin{proof}
Putting $x=12t$ in (\ref{3moments}) and applying formula (15.6.1) from \cite{olver} we get
\begin{align}
a_{n}&=\frac{9\cdot 12^n}{2\pi}\int_{0}^{1}\frac{t^{n-1/2}(1-t)^{3/2}}{(1+3t)^2}\,dt\\
&=\frac{27(2n)!3^n}{8n!(n+2)!}\,{}_{2}F_{1}\!\left(\left.2,n+1/2;\,n+3\,\right|-3\right).
\label{3avoiding3hypergeo}
\end{align}

From (15.8.2) in \cite{olver} and from the identities
\[\frac{\Gamma(n-3/2)}{\Gamma(n+1/2)}=\frac{4}{(2n-3)(2n-1)},\qquad
\frac{\Gamma(3/2-n)}{\Gamma(5/2)}=\frac{(-2)^{n+1}(2n-1)}{3(2n-1)!!},
\]
we have
\[
{}_{2}F_{1}\!\left(\left.2,n+1/2;\,n+3\,\right|-3\right)
=\frac{4(n+2)!}{9n!(2n-1)(2n-3)}\,{}_{2}F_{1}\!\left(\left.2,-n;\,5/2-n\,\right|-1/3\right)
\]
\[
+\frac{(-2)^{n+1}(n+2)!(2n-1)}{3^{n+3/2}(2n-1)!!}\,{}_{2}F_{1}\!\left(\left.n+1/2,-3/2;\,n-1/2\,\right|-1/3\right).
\]
Since
\[
{}_{2}F_{1}\!\left(\left.2,-n;\,5/2-n\,\right|z\right)
=\sum_{k=0}^{n}(k+1)z^k\prod_{i=0}^{k-1}\frac{n-i}{n-5/2-i}
\]
and
\[
{}_2F_{1}\!\left(\left.n+1/2,-3/2;\,n-1/2\,\right|z\right)=\frac{(2n-2nz-2z-1)\sqrt{1-z}}{2n-1}
\]
(see formula (15.4.9) in \cite{olver}), we obtain
\[
{}_{2}F_{1}\!\left(\left.2,n+1/2;\,n+3\,\right|-3\right)
=\frac{n!(n+2)!(8n-1)(-4)^{n+1}}{(2n)!3^{n+3}}
\]
\[
+\frac{4(n+1)(n+2)}{9(2n-1)(2n-3)}\sum_{k=0}^{n}\frac{k+1}{(-3)^k}\prod_{i=0}^{k-1}\frac{n-i}{n-5/2-i},
\]
which leads to~(\ref{3avoiding3suma}).

For the second formula we apply the identity
\[
{}_{2}F_{1}\!\left(\left.2,b;c\,\right|z\right)(1-z)
=(bz-z-c+2)\,{}_{2}F_{1}\!\left(\left.1,b;c\,\right|z\right)
+c-1,
\]
see (15.5.11) in \cite{olver}, to (\ref{3avoiding3hypergeo})
and get
\[
_{2}F_{1}\!\left(\left.2,n+1/2;\,n+3\,\right|-3\right)=
\frac{1-8n}{8}\,{}_{2}F_{1}\!\left(\left.1,n+1/2;\,n+3\,\right|-3\right)+\frac{n+2}{4}.
\]

Applying formula (123), page 462, from \cite{prudnikov3}:
\[
{}_{2}F_{1}\!\left(\left.1,b;\,m+1\,\right|z\right)=\frac{m!}{z^m(b-1)\ldots(b-m)}
\left((1-z)^{m-b}-\sum_{k=0}^{m-1}\frac{z^k}{k!}\prod_{i=0}^{k-1}(b+i-m)\right),
\]
with $b=n+1/2$, $m=n+2$, $z=-3$,
and using the identity
\[
4^{n+1} n!(n+1/2-1)\ldots(n+1/2-n-2)
=3(2n)!
\]
we get (\ref{3avoiding3sumb}).
\end{proof}

\end{document}